\documentclass{amsart}
\usepackage{amssymb}
\usepackage{graphicx}

\newtheorem{thm}{Theorem}
\newtheorem{cor}{Corollary}
\newtheorem{lem}{Lemma}

\newtheorem{rem}{Remark}

\newtheorem{conj}{Conjecture}
\theoremstyle{definition}
\newtheorem{example}[equation]{Example}
\newtheorem{prob}[equation]{Problem}

\newcommand{\real}{{\operatorname{Re}\,}}

\newcommand{\A}{{\mathcal A}}

\newcommand{\U}{{\mathcal U}}
\newcommand{\es}{{\mathcal S}}

\newcommand{\D}{{\mathbb D}}

\def\be{\begin{equation}}
\def\ee{\end{equation}}

\newcommand{\bee}{\begin{enumerate}}
\newcommand{\eee}{\end{enumerate}}

\newcommand{\blem}{\begin{lem}}
\newcommand{\elem}{\end{lem}}
\newcommand{\bthm}{\begin{thm}}
\newcommand{\ethm}{\end{thm}}
\newcommand{\bcor}{\begin{cor}}
\newcommand{\ecor}{\end{cor}}
\newcommand{\beg}{\begin{example}}
\newcommand{\eeg}{\end{example}}
\newcommand{\begs}{\begin{examples}}
\newcommand{\eegs}{\end{examples}}
\newcommand{\bdefe}{\begin{defin}}
\newcommand{\edefe}{\end{defin}}
\newcommand{\bprob}{\begin{prob}}
\newcommand{\eprob}{\end{prob}}
\newcommand{\bei}{\begin{itemize}}
\newcommand{\eei}{\end{itemize}}

\newcommand{\bcon}{\begin{conj}}
\newcommand{\econ}{\end{conj}}
\newcommand{\bcons}{\begin{conjs}}
\newcommand{\econs}{\end{conjs}}
\newcommand{\bprop}{\begin{propo}}
\newcommand{\eprop}{\end{propo}}
\newcommand{\br}{\begin{rem}}
\newcommand{\er}{\end{rem}}
\newcommand{\brs}{\begin{rems}}
\newcommand{\ers}{\end{rems}}
\newcommand{\bo}{\begin{obser}}
\newcommand{\eo}{\end{obser}}
\newcommand{\bos}{\begin{obsers}}
\newcommand{\eos}{\end{obsers}}
\newcommand{\bpf}{\begin{pf}}
\newcommand{\epf}{\end{pf}}
\newcommand{\ba}{\begin{array}}
\newcommand{\ea}{\end{array}}
\newcommand{\beq}{\begin{eqnarray}}
\newcommand{\beqq}{\begin{eqnarray*}}
\newcommand{\eeq}{\end{eqnarray}}
\newcommand{\eeqq}{\end{eqnarray*}}

\begin{document}

\title[Second Hankel determinant for the classes $\U$ and $\es$]{Two types of the second Hankel determinant for the class $\boldsymbol\U$ and the general class $\boldsymbol\es$}

\author[M. Obradovi\'{c}]{Milutin Obradovi\'{c}}
\address{Department of Mathematics,
Faculty of Civil Engineering, University of Belgrade,
Bulevar Kralja Aleksandra 73, 11000, Belgrade, Serbia.}
\email{obrad@grf.bg.ac.rs}

\author[N. Tuneski]{Nikola Tuneski}
\address{Department of Mathematics and Informatics, Faculty of Mechanical Engineering, Ss. Cyril and
Methodius
University in Skopje, Karpo\v{s} II b.b., 1000 Skopje, Republic of North Macedonia.}
\email{nikola.tuneski@mf.ukim.edu.mk}

\subjclass{30C45, 30C55}

\keywords{second order Hankel determinant, class $\U$, classes of univalent functions}

\begin{abstract}
In this paper we determine the upper bounds of the Hankel determinants of special type $H_{2}(3)(f)$ and $H_{2}(4)(f)$ for the class of univalent functions and for the class $\U$ defined by
\[ \U=\left\{ f\in\A : \left|\left[\frac{z}{f(z)}\right]^2 f'(z)-1 \right|<1,\, z\in\D \right\}, \]
where $\A$ is the class of functions analytic in the unit disk $\D$ and normalized such that $f(z)=z+a_2z^2+\cdots$.
\end{abstract}

\maketitle

\section{Introduction and preliminaries}

Let class ${\mathcal A}$ consists of functions analytic in the unit disk $\D:= \{ |z| < 1 \}$
and  are normalized such that
\begin{equation}\label{eq-1}
f(z)=z+a_2z^2+a_3z^3+\cdots,
\end{equation}
i.e., $f(0)=0= f'(0)-1$;
and ${\mathcal S}$ be the class of functions from ${\mathcal A}$ that are univalent in $\D$.

In his paper \cite{zaprawa} Zaprawa considered the following Hankel determinant of second order defined for the coefficients of the function given by \eqref{eq-1}
\[ H_2(n)(f) =  \left |
        \begin{array}{cc}
        a_{n} & a_{n+1}\\
        a_{n+1}&a_{n+2}
        \end{array}
        \right | = a_na_{n+2}-a_{n+1}^2,
\]
for the case when $n=3$. The author studied the upper bound of $|H_2(3)(f)|=|a_3a_5-a_4^2|$ in the cases when $f$ from $\A$ is starlike ($\real [zf'(z)/f(z)]>0$, $z\in\U$), convex ($\real [1+zf''(z)/f'(z)]>0$, $z\in\U$), and with bounded turning ($\real f'(z)>0$, $z\in\U$). These types of functions were studied separately, under the condition that the functions are missing their second coefficient, i.e, $a_2=0$. For the general class $\es$, he proved that $|H_2(3)(f)|>1$. In \cite{OT-2022-3} the authors gave sharp bounds of the modulus of the second Hankel determinant of type $H_2(2)$ of inverse coefficients for various classes of univalent functions.

Another interesting subclass of $\es$ attracting significant interest in the past two decades is
\[ \U=\left\{ f\in\A : \left|\left[\frac{z}{f(z)}\right]^2 f'(z)-1 \right|<1,\, z\in\D \right\}. \]
More details can be found in \cite {OP_2011} and Chapter 12 from \cite{DTV-book}.

\medskip

The object of this paper is to find upper bounds (preferably sharp) of the modulus of the Hankel determinants $H_2(3)(f)=a_3a_5-a_4^2$ and $H_2(4)(f)=a_4a_6-a_5^2$ for the class $\U$, as well as for the general class $\es$.

\medskip

\section{Class $\U$}
For the functions $f$ from the class $\U$ in \cite{OP-2019}, as a pert of the proof of Theorem 1, the following was proven
\begin{equation}\label{repr}
\frac{z}{f(z)} = 1-a_2z-z\omega(z),
\end{equation}
where $|\omega(z)|\le|z|<1$ and $|\omega'(z)|\le1$ for all $z\in\D$, and additionally, for $\omega(z)=c_1z+c_2z^2+\cdots,$
 \begin{equation}\label{eq1.2}
 |c_1|\leq 1,\quad |c_2|\leq \frac{1}{2}(1-|c_1|^2)\quad   \text{and} \quad   |c_3|\leq \frac{1}{3}\left[1-|c_1|^2-\frac{4|c_2|^2}{1+|c_1|}\right].
 \end{equation}
 In a similar way, since $|\omega'(z)|\leq 1$ one can verify that
 \[  |c_4|\leq \frac{1}{4}(1-|c_1|^2-4|c_2|^2).  \]

Further, from \eqref{repr}, we have
\[ z = f(z)\left[1-\left(a_2z+c_1z^2+c_2z^3+\cdots\right)\right], \]
and after equating the coefficients,
\begin{equation}\label{coef}
\begin{split}
a_3&= c_1+a^2_2,\\
a_4&= c_2+2a_2c_1+a_2^3,\\
a_5&= c_3+2a_2c_2+c_1^2+3a_2^2c_1+a_2^4\\
a_6&= c_4+2a_2c_3+2c_1c_2+3a_2^2c_2+3a_2c_1^2+4a_2^3c_1+a_2^5.
\end{split}
\end{equation}

\medskip

Now we can prove the estimates for the class $\U$.

\begin{thm}\label{thm-1}
Let $f\in\U$. Then
\begin{itemize}
  \item[(a)] $|H_2(3)(f)|\le 1$ if $a_2=0$, and the result is sharp due to the function $f(z)=\frac{z}{1-z^2}=z+z^3+z^5+\cdots$.
  \item[(b)] $|H_2(3)(f)|\le 1.4846575\ldots$ for every $f\in\U$.
\end{itemize}
\end{thm}

\begin{proof}
Using \eqref{coef}, after some calculations we receive
\[ H_2(3)(f) = a_3a_5-a_4^2 = (c_1+a_2^2)c_3-2a_2c_1c_2+c_1^3-c_2^2,\]
and from here
\begin{equation}\label{eq-p1-1}
|H_2(3)(f)| \le |c_1+a_2^2||c_3|+2|a_2||c_1||c_2|+|c_1|^3+|c_2|^2.
\end{equation}
\begin{itemize}
  \item[(a)] If $a_2=0$, from \eqref{eq-p1-1} we receive
\[  |H_2(3)(f)| \le |c_1||c_3|+|c_1|^3+|c_2|^2,\]
and using \eqref{eq1.2},
\[
\begin{split}
|H_2(3)(f)| &\le |c_1|\cdot \frac{1}{3}\cdot \left[1-|c_1|^2-\frac{4|c_2|^2}{1+|c_1|}\right]+|c_1|^3+|c_2|^2\\
&= \frac13\left(|c_1|-|c_1|^3\right) + \frac{3-|c_1|}{3(1+|c_1|)}|c_2|^2+|c_1|^3\\
&\le \frac13|c_1|+\frac23|c_1|^3 + \frac{3-|c_1|}{3(1+|c_1|)} \frac14 (1- |c_1|^2)^2\\
&= \frac{1}{12}\left( 3-2|c_1|^2+12|c_1|^3-|c_1|^4 \right)\equiv h_1(|c_1|),
\end{split}
\]
where $h_1(t)=\frac{1}{12}\left( 3-2t^2+12t^3-t^4 \right)$ and $t=|c_1|\le1$ (see \eqref{eq1.2}). Now, $h_1'(t)=-\frac{1}{3}c(1-9c+c^2)$ vanishes in only one point on the interval $(0,1)$ and that is a minimum of $h_1$ on the interval since $h_1(t)<0$ for small enough positive numbers (let say for $t=0.1$). Therefore
\[ \max \{h_1(t):t\in[0,1]\} = \max \{h_1(0),h_1(1)\} = h_1(1)=1, \]
i.e., $|H_2(3)(f)|\le 1$. The sharpness of the estimate follows from the function $f(z)=\frac{z}{1-z^2}$ with $a_2=a_4=0$ and $a_3=a_5=1$.
\medskip

  \item[(b)] Since $\U\subset\es$, we have $|a_2|\le2$ and $|a_3|=|c_1+a_2^2|\le3$, from \eqref{eq-p1-1} we have
  \[ |H_2(3)(f)|\le 3|c_3|+4|c_1||c_2|+|c_1|^3+|c_2|^2 \equiv \varphi_1(|c_1|,|c_2|,|c_3|), \]
  where $\varphi_1(x,y,z)=3z+4xy+x^3+y^2$ with due to \eqref{eq1.2},
 \[   0\le x\le1, \quad 0\le y\le \frac12(1-x^2), \quad 0\le z\le \frac13\left( 1-x^2-\frac{4y^2}{1+x} \right).\]
  It evident that
  \[
  \begin{split}
  \varphi_1(x,y,z) &\le 3\cdot \frac13\left( 1-x^2-\frac{4y^2}{1+x}\right) +4xy+x^3+y^2 \\
  &= 1-x^2+\left(4-\frac{4}{1+x}\right)y^2 -3y^2+4xy+x^3\\
    &\le 1-x^2+ \frac{4x}{1+x}\cdot \frac14 \left(1-x^2\right)^2 -3y^2+4xy+x^3\\
    &= 1+x-2x^2+x^4+4xy-3y^2\equiv \psi(x,y).
    \end{split}
    \]
    It remains to find the maximal value of $\psi$ on the domain $\Omega_1=\Big\{(x,y):0\le x\le1,\, 0\le y\le \frac12(1-x^2)\Big\}$.

    Since $\psi'_y(x,y)=4x-6y$ vanishes for $x=\frac32y$, and $\psi'_x(3y/2,y)= 1-2y+\frac{27 }{2}y^3$ vanishes only for $y=-0.535\ldots$ we realize that $\psi$ attains its maximal value on the boundary of $\Omega_1$. Finally, when $x=0$ or $x=1$, the maximum is 1, while for $y=0$, maximum is $1.1295\ldots$ for $x=0.26959\ldots$, and for $y=\frac12(1-x^2)$, maximum is $1.4846575\ldots$, for $x=0.6618\ldots$. This completes the proof.
\end{itemize}
\end{proof}

\begin{thm}\label{thm-2}
Let $f\in\U$ and $a_2=0$. Then $|H_2(4)(f)|\le 1$ and the estimate is sharp due to the function $f(z)=\frac{z}{1-z^2}=z+z^3+z^5+z^7+\cdots$.
\end{thm}

\begin{proof}
If $f\in\U$ and $a_2=0$, then from \eqref{coef} we receive
\[ a_4=c_2,\quad a_5=c_3+c_1^2,\quad c_6=c_4+2c_1c_2,  \]
and further,
\[ H_2(4)(f) = a_4a_6-a_5^2 = c_2c_4+2c_1c_2^2-c_3^2-2c_1^2c_3+c_1^4,\]
and using \eqref{eq1.2},
\[
\begin{split}
&\quad |H_2(4)(f)| \\
&\le |c_2||c_4|+2|c_1||c_2|^2+|c_3|^2+2|c_1|^2|c_3|+|c_1|^4\\
&\le \frac12(1-|c_1|^2)\cdot\frac14(1-|c_1|^2-4|c_2|^2)+2|c_1||c_2|^2+\frac19\left( 1-|c_1|^2-\frac{4|c_2|^2}{1+|c_1|} \right)^2 \\
&\quad+ 2|c_1|^2\cdot \frac13\left( 1-|c_1|^2-\frac{4|c_2|^2}{1+|c_1|} \right) +|c_1|^4\\
&= A|c_2|^4+B|c_2|^2+C \equiv h_2(|c_2|),
\end{split}
\]
where $h_2(t)=At^4+Bt^2+C$,
\[
\begin{split}
A&=\frac{16}{9(1+|c_1|)^2},\\
B&=2|c_1|-\frac{1}{2} (1-|c_1|^2)-\frac{8}{9} (1-|c_1|)-\frac{8}{3} \frac{|c_1|^2}{1+|c_1|},\\
C&=\frac{17}{72}(1-|c_1|^2)^2+\frac23|c_1|^2(1-|c_1|^2)+|c_1|^4,
\end{split}
\]
with $A>0$, $0\le |c_2|\le \frac12(1-|c_1|^2)$ and $|c_1|\le1$. Therefore, $h_2$ attains its maximal value on the boundary, i.e.,
\[ \max h_2(|c_2|) = \max \left\{ h_2(0), h_2\left(\frac12(1-|c_1|^2) \right) \right\}. \]

\medskip

Now, let note that $h_2(0)=C\equiv g_1(|c_1|)$, where $g_1(t)=\frac{1}{72}(41t^4+14t^2+17)$, has a maximal value 1 when $0\le t=|c_1|\le1$, attained for $t=1$.

\medskip

Further, let $g_2(|c_1|)\equiv h_2\left(\frac12(1-|c_1|^2) \right)$, where
\[ g_2(t)=\frac{1}{72} (17 t^6-12 t^5+38 t^4-24 t^3+17 t^2+36t),\]
$0\le t=|c_1|\le1$. In order to complete the proof of the theorem it is enough to show that this function is increasing on the interval $[0,1]$, which will lead to the conclusion that $h_2\left(\frac12(1-|c_1|^2) \right)=g_2(|c_1|)\le g_2(1)=1$.

Indeed, $g_2'''(t)=\frac{1}{72} \left(1020 t^2-288 t+228\right)>0$ for all $t\in[0,1]$, meaning that $g_2''(t)=\frac{1}{72}\left(340 t^3-144 t^2+228 t-48\right)$ is increasing on the same interval. Since $g_2''(0)<0$ and $g_2''(1)>0$, there is only one real solution of $g_2''(t)=0$ on $[0,1]$, i.e., only one local extreme (minimum) on $[0,1]$ for $t_*=0.22554\ldots$ with value $g_2'(t_*)=39.028\ldots>0$. Thus,  $g_2'(t)>0$ for all $t\in[0,1]$.
\end{proof}

Theorem \ref{thm-1}(a) and Theorem \ref{thm-2} are motivation for the following conjecture for the functions from  $\U$ with missing second coefficient.

\begin{conj}
Let $f\in\U$ and $a_2=0$. Then $|H_2(n)(f)|=|a_na_{n+2}-a_{n+1}^2|\le 1$ for any integer $n\ge3$. The estimate is a sharp due to the function $f(z)=\frac{z}{1-z^2}= \sum_{n=1}^\infty z^{2n-1}$.
\end{conj}

\medskip

\section{General class $\es$}

For obtaining the estimates of the modulus of $H_2(3)(f)$ for the general class $\mathcal{S}$  we will use method based on Grunsky coefficients based on the results and notations given in the book of N.A. Lebedev (\cite{Lebedev}) as follows.

Let $f \in \mathcal{S}$ and let
\[
\log\frac{f(t)-f(z)}{t-z}=\sum_{p,q=0}^{\infty}\omega_{p,q}t^{p}z^{q},
\]
where $\omega_{p,q}$ are the Grunsky's coefficients with property $\omega_{p,q}=\omega_{q,p}$.
For those coefficients the next Grunsky's inequality (\cite{duren,Lebedev}) holds:
\be\label{eq-24}
\sum_{q=1}^{\infty}q \left|\sum_{p=1}^{\infty}\omega_{p,q}x_{p}\right|^{2}\leq \sum_{p=1}^{\infty}\frac{|x_{p}|^{2}}{p},
\ee
where $x_{p}$ are arbitrary complex numbers such that last series converges.

Further, it is well-known that if the function $f$ given by \eqref{eq-1}
belongs to $\mathcal{S}$, then also
\be\label{eq-25}
\tilde{f_{2}}(z)=\sqrt{f(z^{2})}=z +c_{3}z^3+c_{5}z^{5}+\cdots
\ee
belongs to the class $\mathcal{S}$. Then, for the function $\tilde{f_{2}}$ we have the appropriate Grunsky's
coefficients of the form $\omega_{2p-1,2q-1}^{(2)}$ and the inequality \eqref{eq-24} has the form:
\be\label{eq-26}
\sum_{q=1}^{\infty}(2q-1) \left|\sum_{p=1}^{\infty}\omega_{2p-1,2q-1}x_{2p-1}\right|^{2}\leq \sum_{p=1}^{\infty}\frac{|x_{2p-1}|^{2}}{2p-1}.
\ee

Here, and further in the paper we omit the upper index (2) in  $\omega_{2p-1,2q-1}^{(2)}$ if compared with Lebedev's notation.

\medskip

If in the inequality \eqref{eq-26} we put $x_{1}=1$ and $x_{2p-1}=0$ for $p=2,3,\ldots$, then we receive
\be\label{eq-27}
|\omega_{11} |^2 +3|\omega_{13}|^2 + 5|\omega_{15} |^2 +7|\omega_{17}|^2\leq 1.
\ee

As it has been shown in \cite[p.57]{Lebedev}, if $f$ is given by \eqref{eq-1} then the coefficients $a_{2}$, $ a_{3}$, $ a_{4}$ and $a_5$ are expressed by Grunsky's coefficients  $\omega_{2p-1,2q-1}$ of the function $\tilde{f}_{2}$ given by
\eqref{eq-25} in the following way:
\be\label{eq-28}
\begin{split}
a_{2}&=2\omega _{11},\\
a_{3}&=2\omega_{13}+3\omega_{11}^{2}, \\
a_{4}&=2\omega_{33}+8\omega_{11}\omega_{13}+\frac{10}{3}\omega_{11}^{3},\\
a_{5}&=2\omega_{35}+8\omega_{11}\omega_{33}+5\omega_{13}^{2}+18\omega_{11}^2\omega_{13}+\frac73\omega_{11}^4,\\
0&= 3\omega_{15}-3\omega_{11}\omega_{13}+\omega_{11}^3-3\omega_{33},\\
0&=\omega_{17}-\omega_{35}-\omega_{11}\omega_{33}-\omega_{13}^{2}+\frac{1}{3}\omega_{11}^{4}.
\end{split}
\ee

We note that in the cited book of Lebedev there exists a typing mistake for the coefficient $a_{5}$. Namely, instead of the term $5\omega_{13}^{2}$, there is $5\omega_{15}^{2}$.

\bthm\label{22-th-2}
Let $f\in\mathcal{S}$ is given by \eqref{eq-1}. Then
\begin{itemize}
  \item[(a)] $|H_2(3)(f)|\le 2.02757\ldots$ if $a_2=0$;
  \item[(b)] $|H_2(3)(f)|\le 4.8986977\ldots$ for every $f\in\es$.
\end{itemize}
\ethm

\begin{proof} From the fifth relation of \eqref{eq-28} we have
\[ \omega_{33} = \omega_{15} -\omega_{11}\omega_{13}+\frac13\omega_{11}^3.\]
This, together with the sixth relation from \eqref{eq-28} brings
\[ \omega_{35} = \omega_{17}-\omega_{11}\omega_{15}+\omega_{11}^2\omega_{13}-\omega_{13}^2. \]
By applying the two expressions from above in the relations for $a_4$ and $a_5$ from \eqref{eq-28} we obtain
\[
\begin{split}
a_4 &= 2\omega_{15}+6\omega_{11}\omega_{13}+4\omega_{11}^3,\\
a_5 &= 2\omega_{17}+6\omega_{11}\omega_{15}+12\omega_{11}^2\omega_{13}+3\omega_{13}^2+5\omega_{11}^4.
\end{split}
\]
Finally, these two relations, together with the relation for $a_3$ from \eqref{eq-28} bring
\begin{equation}\label{eq-q}
\begin{split}
H_2(3)(f) &= a_3a_5-a_4^2\\
&= 2(2\omega_{13}+3\omega_{11}^2)\omega_{17} - 12\omega_{11}\omega_{13}\omega_{15} - 3\omega_{11}^2\omega_{13}^2 + 6\omega_{13}^3\\
&\quad  - 2\omega_{11}^4\omega_{13} + 2\omega_{11}^3\omega_{15} - \omega_{11}^6 - 4 \omega_{15}^2.
\end{split}
\end{equation}

\begin{itemize}
  \item[(a)] If $a_2=2\omega_{11}=0$, then $\omega_{11}=0$, and we receive
\[H_2(3)(f) = 4\omega_{13}\omega_{17}+6\omega_{13}^3-4\omega_{15}^2,\]
with the following constraints over $\omega_{13}$, $\omega_{15}$ and $\omega_{17}$ obtained from \eqref{eq-27}:
\[ |\omega_{13}|\le \frac{1}{\sqrt3}, \quad |\omega_{15}|\le \frac{1}{\sqrt5}\sqrt{1-3|\omega_{13}|^2},\]
and
\[|\omega_{17}|\le \frac{1}{\sqrt7}\sqrt{1-3|\omega_{13}|^2-5|\omega_{15}|^2}.\]
So,
\[
\begin{split}
|H_2(3)(f)| &= 4|\omega_{13}||\omega_{17}|+6|\omega_{13}|^3+4|\omega_{15}|^2 \\
& \le  \frac{4}{\sqrt7} |\omega_{13}|\sqrt{1-3|\omega_{13}|^2-5|\omega_{15}|^2}+6|\omega_{13}|^3+4|\omega_{15}|^2 \\
&= \psi_1(|\omega_{13}|,|\omega_{15}|),
\end{split}
\]
where $\psi_1(y,z) = \frac{4}{\sqrt7} y\sqrt{1-3y^2-5z^2}+6y^3+4z^2$ with $0\le y=|\omega_{13}|\le\frac{1}{\sqrt3}$, $0\le z=|\omega_{15}|\le \frac{1}{\sqrt5}\sqrt{1-3y^2}$.
It remains to find upper bound of the function $\psi_1(y,z)$ on its domain
\[ \Omega=\left\{(y,z): 0\le y\le\frac{1}{\sqrt3}, \, 0\le z\le \frac{1}{\sqrt5}\sqrt{1-3y^2}\right\}. \]
Not being able to do better and leaving the sharp bound as an open problem, we continue with what is easy to get:
\[
\begin{split}
\psi_1(y,z) &\le \frac{4}{\sqrt7}y+6y^3+\frac45(1-3y^2)= \frac45+\frac{4}{\sqrt7}y - \frac{12}{5}y^2+6y^3\\
&\le \frac{4}{\sqrt{21}}+\frac{2}{\sqrt3} = 2.02757\ldots,
\end{split}
\]
obtained for $y=\frac{1}{\sqrt3}$.

\medskip

  \item[(b)] In the general case, if $a_2\neq0$, since $|a_2|\le2$ and $|c_1+a_2^2|=|a_3|\le3$, from \eqref{eq-q} we get
\[
\begin{split}
|H_2(3)(f)| &= 6|\omega_{17}|+12|\omega_{11}||\omega_{13}||\omega_{15}| +3|\omega_{11}|^2|\omega_{13}|^2 \\
&\quad +6|\omega_{13}|^3+2|\omega_{11}|^4|\omega_{13}| + 2|\omega_{11}|^3|\omega_{15}| \\
&\quad +|\omega_{11}|^6 +4|\omega_{15}|^2\\
&\le 6\cdot \frac{1}{\sqrt7}\sqrt{1-|\omega_{11}|^2-3|\omega_{13}|^2-5|\omega_{15}|^2}+12|\omega_{11}||\omega_{13}||\omega_{15}|  \\
&\quad +3|\omega_{11}|^2|\omega_{13}|^2+6|\omega_{13}|^3+2|\omega_{11}|^4|\omega_{13}| \\
&\quad + 2|\omega_{11}|^3|\omega_{15}| +|\omega_{11}|^6 +4|\omega_{15}|^2 \\
&= \psi_2(|\omega_{11}|,|\omega_{13}|,\omega_{15}|),
\end{split}
\]
where 
\[\begin{split}
\psi_2(x,y,z) &= \frac{6}{\sqrt7}\sqrt{1-x^2-3y^2-5z^2}+12xyz +3x^2y^2 \\
&\quad +6y^3+2x^4y + 2x^3z +x^6 +4z^2
\end{split}
\]
with $0\le x=|\omega_{11}|\le1$, $0\le y=|\omega_{13}|\le\frac{1}{\sqrt3}\sqrt{1-x^2}$, $0\le z=|\omega_{15}|\le \frac{1}{\sqrt5}\sqrt{1-x^2-3y^2}$.
Similarly as in the part (a), finding upper bound of the function $\psi_2(x,y,z)$ on its domain
\[ 
\begin{split}
&\left\{(x,y,z): 0\le x\le1,\, 0\le y\le\frac{1}{\sqrt3}\sqrt{1-x^2}, \right.\\
& \quad  \left.0\le z\le \frac{1}{\sqrt5}\sqrt{1-x^2-3y^2}\right\},
\end{split}
\]
is still an open problem, even though analysis suggest that it is 1. Easy way arround, leading to a non-sharp upper bound is:
\[
\begin{split}
\psi_2(x,y,z) &\le \frac{6}{\sqrt7}\sqrt{1-x^2}+12xyz +3x^2y^2 \\
&\quad +6y^3+2x^4y + 2x^3z +x^6 +4z^2,
\end{split}
\]
which after applying $y\le\frac{1}{\sqrt3}\sqrt{1-x^2}$ and $z\le\frac{1}{\sqrt5}\sqrt{1-x^2}$ leads to
\[
\begin{split}
\psi_2(x,y,z) &\le \frac{6}{\sqrt7}\sqrt{1-x^2}+\frac{12}{\sqrt{15}}x(1-x^2) +\frac{3}{15}(1-x^2)^2 \\
&\quad +\frac{6}{3\sqrt3}(1-x^2)\sqrt{1-x^2}+\frac{2}{\sqrt3}x^4\sqrt{1-x^2} \\
&\quad + \frac{2}{\sqrt5}x^3\sqrt{1-x^2} +x^6 +\frac45(1-x^2)^2\equiv h_*(x).
\end{split}
\]
Numerical computations show that this function has maximal value $4.8986977\ldots$ obtained for $x=0.3945667\ldots$.
\end{itemize}
\end{proof}

\end{document}